\documentclass[11pt]{amsart}
\usepackage{graphicx}
\usepackage{amsmath,amsfonts,amssymb,latexsym,amsthm,amscd}
 \newtheorem{theorem}{Theorem}[section]
 \newtheorem{lemma}[theorem]{Lemma}

 \newtheorem{remark}[theorem]{Remark}

\begin{document}

\title{The Extended Mapping Class Group Can Be Generated by Two Torsions}

\author{Xiaoming Du}
\address{South China University of Technology,
  Guangzhou 510640, P.R.China}
\email{scxmdu@scut.edu.cn}



\thanks{This research is supported by NSFC (Grant No. 11401219). 
The author wish to thank BICMR for its hospitality and thank Boju Jiang who asked
me about the problem of finding lower bounds of the orders and
the numbers of torsion generators for the mapping class groups. 
The author also wish to thank Dingsen Yan who teach me how to draw pictures 
with the metapost package of the latex.}

\maketitle

\begin{abstract}
  Let $S_g$ be the closed oriented surface of genus $g$ and let
  $\text{Mod}^{\pm}(S_g)$ be the extended mapping class group of $S_g$.
  When the genus is at least 5, we prove that $\text{Mod}^{\pm}(S_g)$
  can be generated by two torsion elements. One of these generators is
  of order $2$, and the other one is of order $4g+2$.
\end{abstract}

\section{Introduction}

Let $S_g$ be the closed oriented surface of genus $g$. The extended mapping class group
$\text{Mod}^{\pm}(S_g)$ is defined as $\text{Homeo}^{\pm}(S_g)/\text{Homeo}_{0}(S_g)$,
the group of homotopy classes of homeomorphisms (including orientation-preserving
ones and orientation-reversing ones) of $S_g$, and the mapping class group $\text{Mod}(S_g)$
is defined by $\text{Homeo}^{+}(S_g)/\text{Homeo}_{0}(S_g)$, the group of
orientation-preserving homotopy classes of homeomorphisms of $S_g$.

For $\text{Mod}(S_g)$, Dehn and Lickorish found independently Dehn twist generating sets
of $\text{Mod}(S_g)$ \cite{De,Li}. Humphries reduced the number of Dehn twist
generators to the lowest bound \cite{Hu}. Wajnryb in \cite{Wa} found that the minimal number
of the generators (not only Dehn twist generators) for the mapping class groups is 2.

We are also interested in the torsion generating set.
McCarthy and Papadopoulos in \cite{MP} proved that $\text{Mod}(S_g)$ can be generated
by infinitely many elements of order 2 when $g \geq 3$.
Luo showed that $\text{Mod}(S_g)$ is generated by $12g+6$ elements of order $2$ when
$g \geq 3$ \cite{Lu}. Brendle and Farb reduced the number of the involution generators to
$6$ \cite{BF} and prove $\text{Mod}(S_g)$ can be generated by three torsion elements.
Kassabov reduced the number of the involution generators to 4 \cite{Ka} for $g \geq 7$.
Korkmaz in \cite{Ko1} proved $\text{Mod}(S_g)$ can be generated by two torsion elements
of order $4g+2$. Monden in \cite{Mo} proved $\text{Mod}(S_g)$ can be generated by $3$ torsion
elements of order 3.

For the extended mapping class group $\text{Mod}^{\pm}(S_g)$, Dehn-Nielsen-Baer
theorem \cite{FM} says that $\text{Mod}^{\pm}(S_g)$ is isomorphic to $\text{Out}(\pi_1(S_g))$,
the outer automorphism group of $\pi_1(S_g)$. Ivanov proved that $\text{Mod}^{\pm}(S_g)$
is the automorphism group of the curve complex \cite{Iv}. Brock and Margalit showed that
$\text{Mod}^{\pm}(S_g)$ is the automorphism group of the pants complex and it is also the
isometry group of the Teichm\"uller space under the Weil-Petersson metric (\cite{BM},\cite{Ma}).
Korkmaz showed that $\text{Mod}^{\pm}(S_g)$
can be generated by 2 elements, one of which is a Dehn twist \cite{Ko1}.
Stukow in \cite{St} proved $\text{Mod}^{\pm}(S_g)$ is generated by 3 elements of order 2.

It is an open problem if $\text{Mod}^{\pm}(S_g)$ can be generated by two torsion elements
(See \cite{Ko2}, Problem 5.3). In this paper, under the condition that the genus of the surface
is at least 5, we answer this question affirmatively:

\begin{theorem}
For $g \geq 5$, the extended mapping class group $\text{Mod}^{\pm}(S_g)$ can be generated
by 2 torsion elements. One of these generators is of order 2 and the other one is of
order $4g+2$.
\end{theorem}

\section{Preliminaries}

\textbf{Notations:}

(a) We use the convention of functional notation, namely, elements of
the mapping class group are applied right to left, i.e. the composition $FG$ means
that $G$ is applied first.

(b) A Dehn twist means a right-hand Dehn twist.

(c) We denote the curves by lower case letters $a$, $b$, $c$, $d$
(possibly with subscripts) and the Dehn twists about them by the corresponding
capital letters $A$, $B$, $C$, $D$. Notationally we do not distinguish a
diffeomorphism/curve and its isotopy class.

We recall the following results (see, for instance, section 3.3, 5.1, 7.5 of \cite{FM}):

\begin{lemma}
For any $\varphi \in \text{Mod}(S_g)$ and any isotopy classes $a, b$
of simple closed curves  in $S_g$ satisfying $\varphi(a)=b$, we have:
$$B=\varphi\,A\,\varphi^{-1}.$$
\end{lemma}

\begin{lemma}
For any $\varphi \in \text{Mod}^{\pm}(S_g)\setminus\text{Mod}(S_g)$ and any isotopy classes $a, b$
of simple closed curves  in $S_g$ satisfying $\varphi(a)=b$, we have:
$$B^{-1}=\varphi\,A\,\varphi^{-1}.$$
\end{lemma}

\begin{lemma}
Let $a, b$ be two simple closed curves on $S_g$. If $a$ is disjoint from $b$, then
$$AB=BA.$$
\end{lemma}

\vspace{0.3cm}

\centerline{\includegraphics[totalheight=4cm]{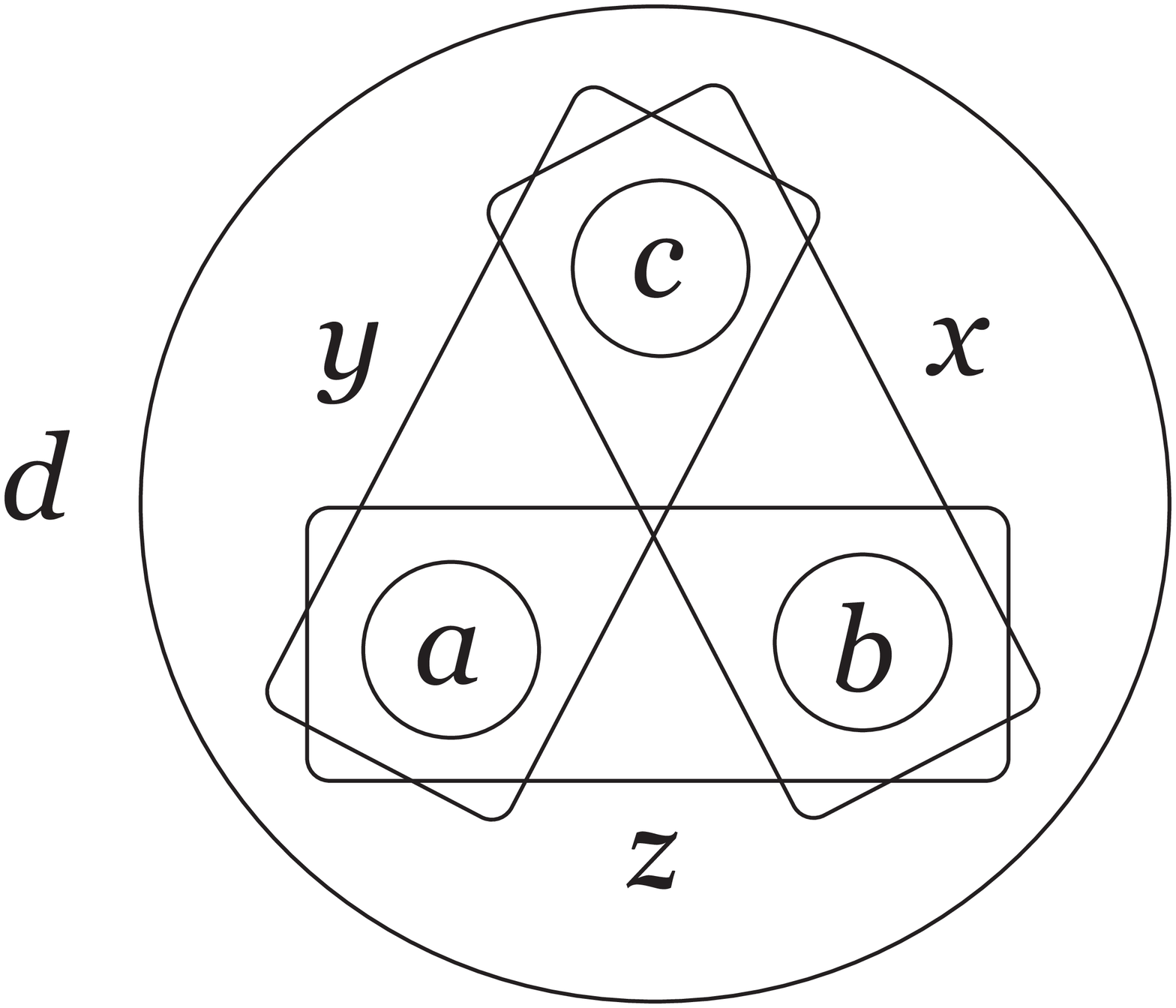}}

\centerline{Figure 1}

\begin{lemma}[Lantern relation]
Let $a, b, c, d, x, y, z$ be the curves showed in Figure 1 on a genus zero
surface with four boundaries. Then
$$ABCD=XYZ.$$
In other words, since $a, b, c$ are disjoint from $x, y, z$, we have
$$D=(XA^{-1})(YB^{-1})(ZC^{-1}).$$
\end{lemma}

The elements having the form of $UV^{-1}$ for some disjoint simple closed
curves $u,v$ play a crucial role in our proof.
The proof of the main result relies on Humphries' theorem:

\begin{theorem}[Humphries]
Let $a_1, a_2, \dots\, a_{2g}, b$ be the curves as on the left-hand side of Figure 2.
Then the mapping class group $\text{Mod}(S_g)$ is generated by $A_i$'s and $B$.
\end{theorem}

\centerline{\includegraphics[totalheight=3.5cm]{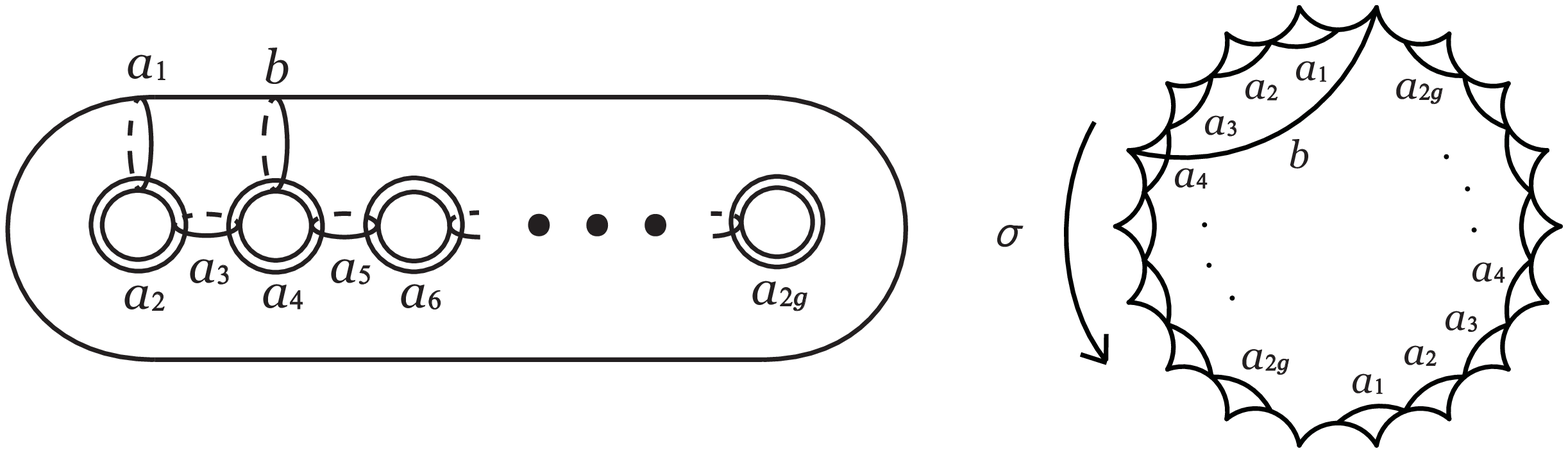}}

\centerline{Figure 2}

Consider the genus $g$ surface as a regular $(4g+2)$-gon whose corresponding
opposite sides are glued together, as indicated on the right-hand side of Figure 2.
We redraw the curves of Humphries' generating set on the right-hand side of Figure 2 as follow.
The set of curves $a_i$'s on the left-hand side of Figure 2 is a chain of simple
closed curves and fills the surface. We find a chain of simple closed curves which
also fills the surface on the right-hand side of Figure 2, identify them with $a_i$'s.
To see the corresponding curve of $b$, notice that $b$ intersects $a_4$ once and is disjoint
from other $a_i$'s. Then find such a curve on the right-hand side of Figure 2.

Look at the $(4g+2)$-gon. There is a natural $2\pi/(4g+2)$ rotation preserving the gluing
way of the $(4g+2)$-gon. This rotation induces a period map $\sigma$ of the genus $g$
surface. Moreover, $\sigma(a_{i})=a_{i+1}$ for $1 \leq i \leq 2g-1$.
Take $a_0=\sigma(a_{2g})$. Then we have $a_1=\sigma(a_0)$.
Under modulo $2g+1$, we have $a_k=\sigma^{k}(a_0)$ for all integer $k$
and $a_{i+1}=\sigma(a_{i})$ still holds for all $i$'s.
Similarly, since $\sigma$ is of order $4g+2$, under modulo $4g+2$,
we take $b_0=b$ and $b_k=\sigma^{k}(b)$. Then we have $\sigma(b_j)=b_{j+1}$ for all $j$'s.

To see back the image of $b_j$'s on the left-hand side of Figure 2,
since the set of $a_k$'s form a chain of simple closed curves filling the surface,
we only need to calculate the geometric intersection number $i(b_j,a_k)$
for each $k$ on the right-hand side of Figure 2,
then find a curve with the same geometric intersection numbers
with $a_k$'s as on the left-hand side of Figure 2.
Figure 3 shows this way to think of the images of $b_j$'s.
Figure 2 and Figure 3 will be used in the proof of Theorem \ref{main} to verify
the disjointness between some $a_i$ and $b_j$.

\vspace{0.3cm}

\centerline{\includegraphics[totalheight=7cm]{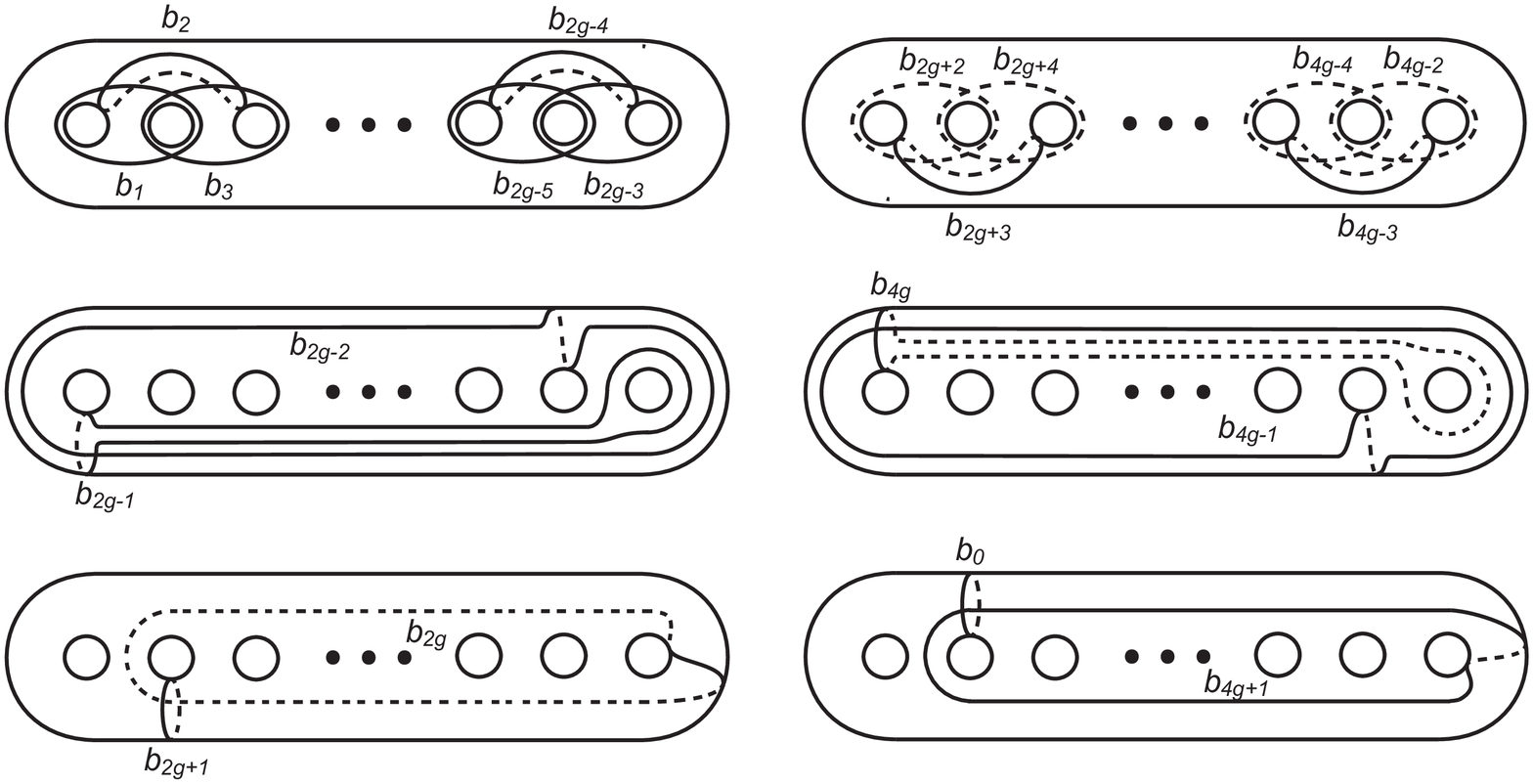}}

\vspace{0.3cm}

\centerline{Figure 3}

\section{The main result and the proof}

Let $\tau$ be the symmetry of the $(4g+2)$-gon as on the left-hand side of Figure 4.
Then $\tau$ induces an orientation-reversing homeomorphism of the genus $g$ surface.
The fixed point set of $\tau$ forms a non-separating curve.
A part of such a non-separating curve is along the axis of the symmetry.
We still denote the homotopy class of such a homeomorphism as $\tau$.
Since $\tau$ preserves the curve $b$, it is easy to check that $(\tau \circ B)^2=Id$.

\vspace{0.3cm}

\centerline{\includegraphics[totalheight=3.5cm]{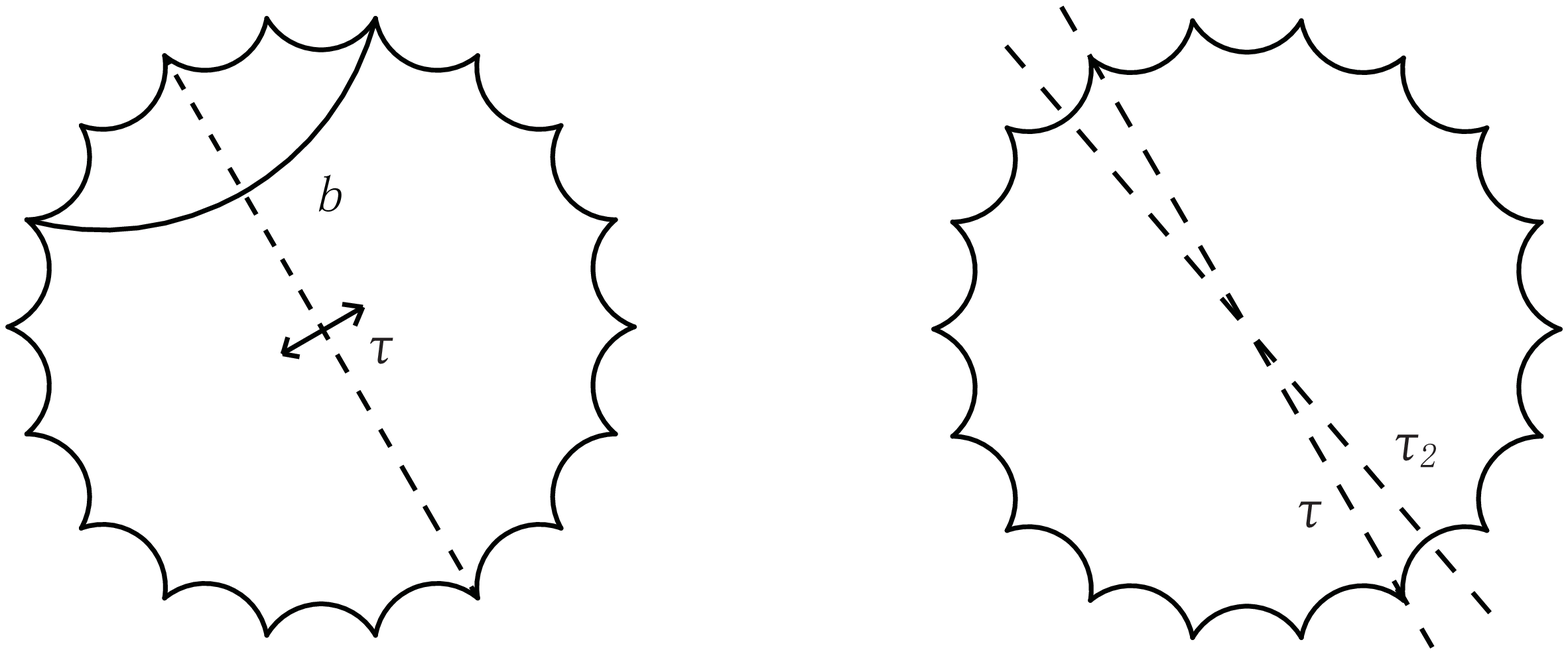}}

\vspace{0.3cm}

\centerline{Figure 4}

\begin{theorem}\label{main}
Let $\sigma$ be the element of order $4g+2$ and $\tau \circ B$ be the element of
order $2$ as we described above. Then for $g \geq 5$ we have
$\text{Mod}^{\pm}(S_g) = \langle \sigma, \tau \circ B \rangle$.
\end{theorem}

\begin{remark}
This generating set can also apply to $\text{Mod}^{\pm}(S_{g,1})$,
i.e., the extended mapping class group of the surface with one marked point.
This is because the element $\sigma$ and $\tau \circ B$ fix the center point of the $4g+2$-gon.
\end{remark}

\begin{remark}
From Theorem \ref{main} we can also deduce that $\text{Mod}^{\pm}(S_g)$ can be
generated by 3 symmetries when $g \geq 5$. The reason is that $\sigma$ is like
the rotation of the $(4g+2)$-gon.
In the dihedral group of the $(4g+2)$-gon, the rotation is the product
of two reflections (see the right-hand side of Figure 4).
Such reflections of $(4g+2)$-gon induce orientation reversing
order 2 maps $\tau, \tau_2$ on the surface.
Hence $\text{Mod}^{\pm}(S_g) = \langle \tau \circ B, \tau, \tau_2 \rangle$.
So we can get another proof of the result in \cite{St} under the condition $g \geq 5$.
Our generating set is different from the generating set in \cite{St}.
The fixed-point sets of two of the symmetries in \cite{St} are separating curves.
The fixed-point sets of our symmetries are non-separating curves.
\end{remark}

\begin{proof}[Proof of Theorem \ref{main}]
Denote the subgroup generated by $\sigma$ and $\tau \circ B$ as $G$.
We prove that $G=\text{Mod}^{\pm}(S_g)$ in four steps:

Step 1. Under modulo $4g+2$, the following two conditions are equivalent:
(1) integers $i,k$ satisfy $k \in \{4, 5, 6, \dots, 4g-2\} \setminus \{2g-2, 2g, 2g+2, 2g+4\}$;
(2) $b_i$ is disjoint from $b_{i+k}$.
Under such conditions we have $B_i\,B_{i+k}^{-1}$ and $B_i^{-1}\,B_{i+k}$ are in $G$.

Step 2. Under modulo $2g+1$ for $m$ and modulo $4g+2$ for $n$, the following two conditions are equivalent:
(1) integers $m,n$ satisfy $m \not \in \{n, n+4\}$;
(2) $a_m$ is disjoint from $b_n$.
Under such conditions we have $A_m\,B_n^{-1}$ and $A_m^{-1}\,B_n$ are in $G$.

Step 3. Using the lantern relation, prove that for all $k$, $A_k \in G$.

Step 4. Finally, $G=\text{Mod}^{\pm}(S_g)$.

The proof of Step 1:

It is obvious that $b_0$ is disjoint from $b_k$ if and only if
$k \in \{4, 5, 6, \dots, 4g-2\} \setminus \{2g-2, 2g, 2g+2, 2g+4\}$ (see Figure 2 and Figure 3).
So do $b_i$ and $b_{i+k}$.
Under such conditions, we first prove that $B_0B_k^{-1} \in G$.
Consider the element $\sigma^{k}\,(\tau \circ B)\,\sigma^{k}\,(\tau \circ B)$.
In the dihedral subgroup of $\text{Mod}^{\pm}(S_g)$,
$\sigma^{k}\,\tau\,\sigma^{k}\,\tau$ is the identity.
After adding the Dehn twist $B$, since $b_k$ is disjoint from $b_0$,
we can easily check that $\sigma^{k}\,(\tau \circ B)\,\sigma^{k}\,(\tau \circ B)$ $=B_{0}B_{k}^{-1}$.
For every integer $i$, conjugate $B_{0}B_{k}^{-1}$ by $\sigma^{i}$.
Then we have $B_iB_{i+k}^{-1}$ is in $G$.
The commutativity of $B_i$ and $B_{i+k}$ promise $B_i^{-1}B_{i+k}$ is in $G$.

The proof of Step 2:

It is obvious that the disjointness between $a_i$ and $b_0$ is equivalent to $i \not \in \{0,4\}$.
So $a_{m}$ is disjoint from $b_n$ is equivalent to $m \not \in \{n,n+4\}$.
From these conditions, we also see that $a_i$ is disjoint from $b_0$ if and only if
$a_{4-i}$ is disjoint from $b_0$.

By the conjugacy relation, for some $i,k,m,n$ satisfying the conditions in Step 1 and Step 2,
if there exists $\varphi \in G$ and a pair of disjoint curves $(b_i,b_{i+k})$
such that $\varphi: (b_i, b_{i+k}) \mapsto (a_m, b_n)$, since $B_iB_{i+k}^{-1}$ is in $G$,
$A_mB_n^{-1}$ is also in $G$. We need to find such a $\varphi \in G$.

The proofs in odd genus case and in even genus case are slightly different.
Suppose first that the genus $g$ is odd. See Figure 5.

\vspace{0.5cm}

\centerline{\includegraphics[totalheight=6cm]{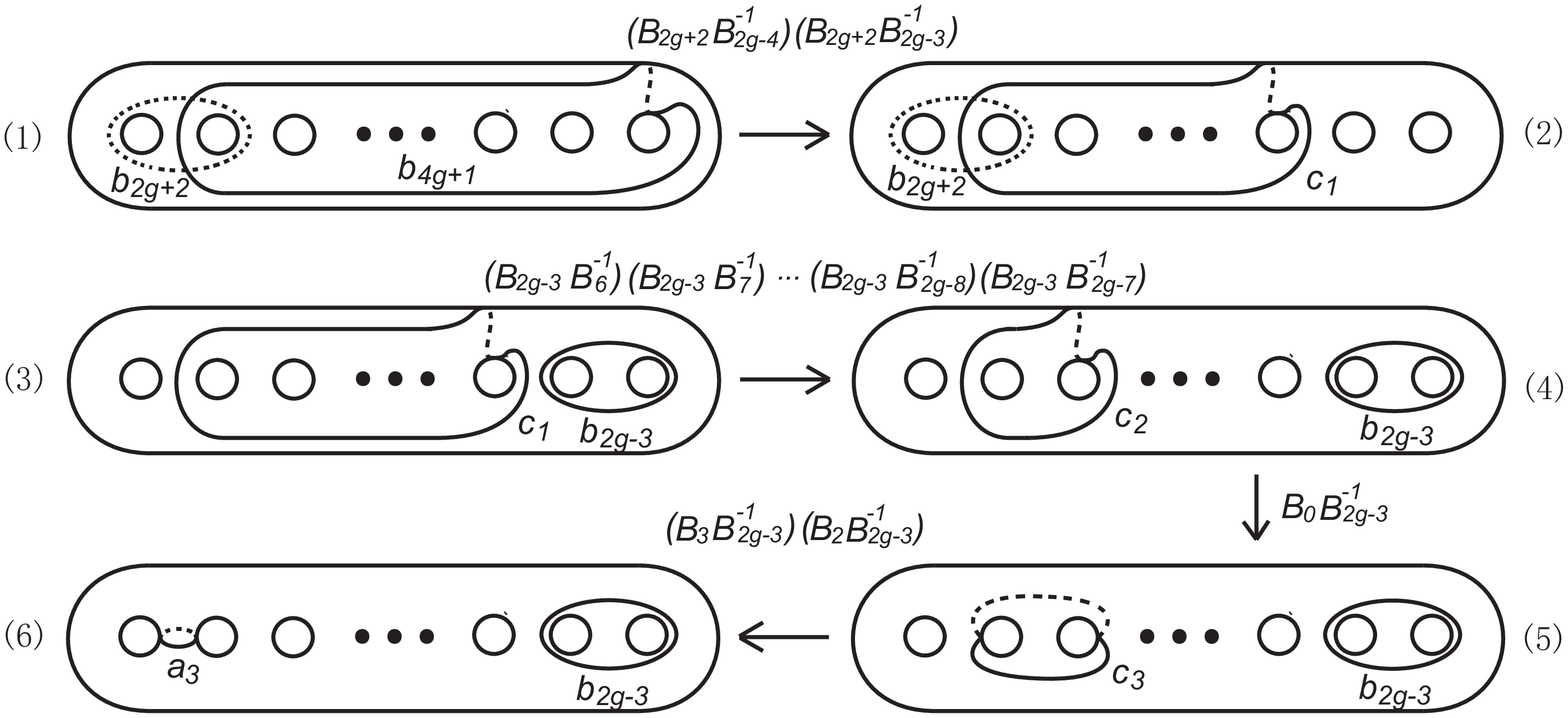}}

\centerline{Figure 5}

\vspace{0.5cm}

By Step 1, $B_{2g+2}B_{4g+1}^{-1}$ is in $G$.
If $g \geq 5$, $b_{2g+2}$ is disjoint from $b_{2g-3}, b_{2g-4}$.
Both $B_{2g+2}B_{2g-4}^{-1}$ and $B_{2g+2}B_{2g-3}^{-1}$ are in $G$.
The element $(B_{2g+2}B_{2g-4}^{-1})$ $\cdot$ $(B_{2g+2}B_{2g-3}^{-1})$
maps the pair of curves $(b_{2g+2},b_{4g+1})$ to the pair of curves $(b_{2g+2},c_1)$ as in Figure 5 (2).
Hence $B_{2g+2}C_1^{-1}$ is in $G$.
Now $B_{2g-3}B_{2g+2}^{-1}$ is also in $G$.
So $B_{2g-3}C_1^{-1}$ is in $G$.

The curve $b_{2g-3}$ is disjoint from $b_{2g-7}$, $b_{2g-8}, \dots, b_7$, $b_6$.
The elements $(B_{2g-3}B_6^{-1})$, $(B_{2g-3}B_7^{-1})\,\dots\,(B_{2g-3}B_{2g-8}^{-1})$,
$(B_{2g-3}B_{2g-7}^{-1})$ are in $G$. Then their product
$(B_{2g-3}B_6^{-1})$ $\cdot$ $(B_{2g-3}B_7^{-1})\,\dots\,(B_{2g-3}B_{2g-8}^{-1})$ $\cdot$ $(B_{2g-3}B_{2g-7}^{-1})$
maps the pair of curves $(b_{2g-3},c_1)$ to the pair of curves $(b_{2g-3},c_2)$ as in Figure 5 (4).
We have $B_{2g-3}C_2^{-1}$ is in $G$.

The curve $b_{2g-3}$ is disjoint from $b_3$ and $b_2$.
The elements $(B_{2g-3}B_2^{-1})$, $(B_{2g-3}B_3^{-1})$ are in $G$.
Their product $(B_3B_{2g-3}^{-1})$ $\cdot$ $(B_2B_{2g-3}^{-1})$ maps the pair of curves
$(b_{2g-3},c_2)$ to the pair of curves $(b_{2g-3},a_3)$.
So we get that $B_{2g-3}A_3^{-1}$ is in $G$. Its inverse $A_3B_{2g-3}^{-1}$ is also in $G$.

Suppose now that $g$ is even. See Figure 6.
\vspace{0.5cm}

\centerline{\includegraphics[totalheight=4cm]{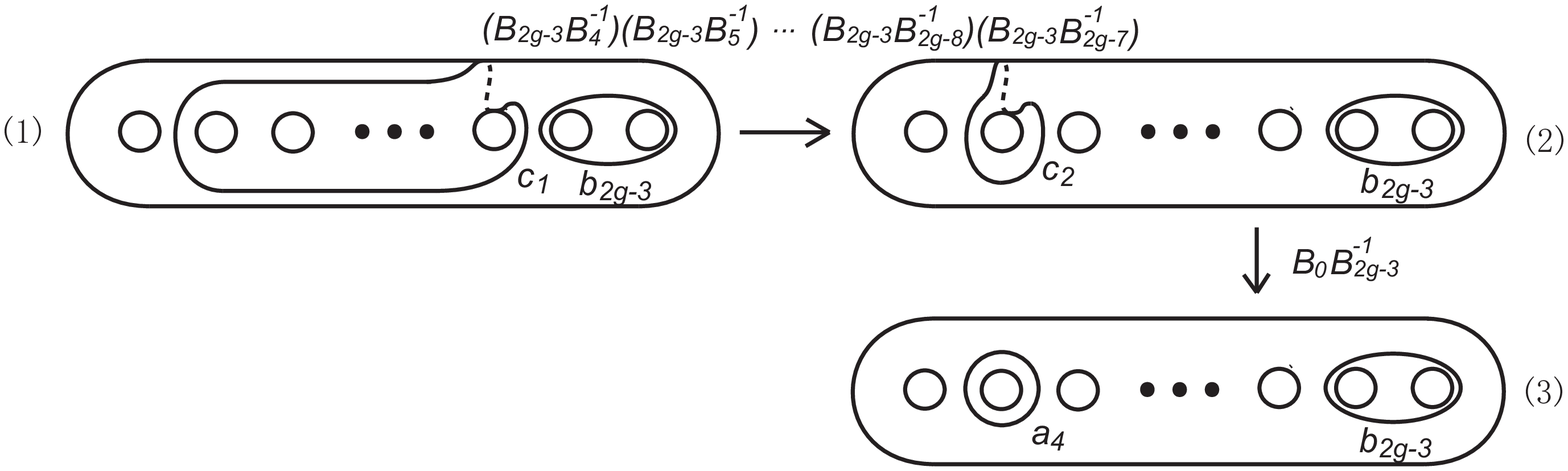}}

\centerline{Figure 6}

\vspace{0.5cm}
In this case, the proof of $B_{2g-3}C_1^{-1} \in G$ is the same as the odd genus case.
The element $(B_{2g-3}B_4^{-1})$ $\cdot$ $(B_{2g-3}B_5^{-1})\,\dots\,(B_{2g-3}B_{2g-8}^{-1})$ $\cdot$ $(B_{2g-3}B_{2g-7}^{-1})$
maps the pair of curves $(b_{2g-3},c_1)$ to the pair of curves $(b_{2g-3},c_2)$ as in Figure 6 (2).
The element $B_0B_{2g-3}^{-1}$ maps the pair of curves $(b_{2g-3},c_2)$ to the pair of curves $(b_{2g-3},a_4)$.
So we get that $B_{2g-3}A_4^{-1}$ is in $G$. Its inverse $A_4B_{2g-3}^{-1}$ is also in $G$.

The element $\sigma^{2g+5}$ maps $(a_3,b_{2g-3})$ to $(a_{2g+8},b_0)$ when $g$ is odd
and maps $(a_4,b_{2g-3})$ to $(a_{2g+9},b_0)$ when $g$ is even.
We have $A_iB_0^{-1}=B_0^{-1}A_i \in G$ for some $i$ where $a_i$ is disjoint from $b_0$.

Notice that $\tau\,\sigma^{k}\,\tau$ maps $a_i$ to $a_{i-k}$,
hence $A_{i-k} = (\tau\,\sigma^{k}\,\tau)\,A_i\,(\tau\,\sigma^{k}\,\tau)^{-1}$.
So as long as $a_{i-k}$ is disjoint from $b_0$, we have\\
$
\begin{array}{lll}
& & (\tau \circ B_0)\,\sigma^{k}\,(\tau \circ B_0)\,(A_iB_0^{-1})\,\sigma^{k} = (B_0^{-1} \circ \tau)\,\sigma^{k}\,(\tau \circ A_i)\,\sigma^{k}\\
&=& B_0^{-1}\,(\tau\,\sigma^{k}\,\tau)\,A_i\,\sigma^{k} = B_0^{-1}\,(A_{i-k}\,\tau\,\sigma^{k}\,\tau)\,\sigma^{k} = B_0^{-1}\,A_{i-k}=A_{i-k}\,B_0^{-1}.
\end{array}
$

Taking all possible $k$ and conjugating by $\sigma^n$,
we have all $A_mB_n^{-1} \in G$ for $m \not \in \{n, n+4\}$.
The commutativity of $A_m$ and $B_n$ promises $A_m^{-1}B_n$ is in $G$.

The proof of Step 3:

See Figure 7. There is a natural lantern lying on the surface,
bounded by $a_1, a_3, a_5$ and $f$.
By the lantern relation, we have $B_0B_2E=A_1A_3A_5F$, or
$A_1=(B_0A_3^{-1})(B_2A_5^{-1})(EF^{-1})$,
where $e$ and $f$ are the curves showed in Figure 7.
The elements $B_0A_3^{-1}$ and $B_2A_5^{-1}$ are the inverses of $A_3B_0^{-1}$ and $A_5B_2^{-1}$ respectively.
By the result of step 2, they are in $G$.
We only need to prove $EF^{-1}$ is also in $G$.

\vspace{0.3cm}

\centerline{\includegraphics[totalheight=2cm]{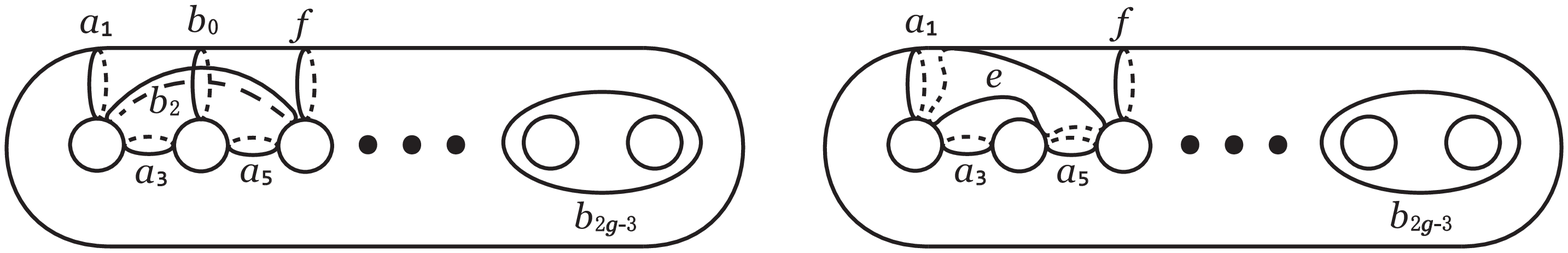}}

\centerline{Figure 7}

\vspace{0.3cm}

When $g \geq 5$, $b_{2g-3}$ is disjoint from $b_0, \dots, b_3, a_1, \dots, a_6, e, f$.
Notice $EF^{-1}=(EB_{2g-3}^{-1})(B_{2g-3}F^{-1})$.
We verify $EB_{2g-3}^{-1}$ and $B_{2g-3}F^{-1}$ are in $G$.
By the previous steps,
The element $(B_{2g-3}B_3^{-1})$ $\cdot$ $(A_6B_{2g-3}^{-1})$ $\cdot$ $(A_5B_{2g-3}^{-1})$ $\cdot$ $(A_4B_{2g-3}^{-1})$
is in $G$ and maps the pair of curves $(b_{2g-3},b_0)$ to $(b_{2g-3},f)$.
Notice $B_{2g-3}B_0^{-1}$ is in $G$. So $B_{2g-3}F^{-1}$ is in $G$.
The element $(A_2B_{2g-3}^{-1})$ $\cdot$ $(A_1B_{2g-3}^{-1})$ $\cdot$ $(A_4^{-1}B_{2g-3})$ $\cdot$ $(B_1B_{2g-3}^{-1})$
maps the pair of curves $(b_{2g-3},a_5)$ to $(b_{2g-3},e)$.
Notice $B_{2g-3}A_5^{-1}$ is in $G$. So $B_{2g-3}E^{-1}$ is in $G$.
We now get $EF^{-1}=(EB_{2g-3}^{-1})(B_{2g-3}F^{-1})$ is in $G$.

The proof of Step 4:

The fact that the elements $A_1$ and $A_1B_0^{-1}$ are in $G$ means that $B_0 \in G$.
Now all the curves $a_i$'s are in the same orbit of $\sigma$. So do $b_j$'s.
So all the $A_i$'s and $B_j$'s are in $G$. These include Humphries' generators of $\text{Mod}(S_g)$.
So $G$ contains $\text{Mod}(S_g)$.
Finally, one of the generator $\tau \circ B$ is an orientation-reversing mapping class.
Hence $G=\text{Mod}^{\pm}(S_g)$.

\end{proof}

\begin{remark}
The idea of Step 2 and Step 3 in the above proof is based on the method in
\cite{Ko1} to prove $\text{Mod}(S_g)$ is generated by two torsions of order $4g+2$.
\end{remark}

\end{document}